\documentclass[12pt]{amsart}

\usepackage{amsmath, amssymb, amsfonts}
\usepackage{enumerate}

\usepackage{fullpage} 
\usepackage{setspace}
\usepackage[
    bibstyle=alphabetic,
    citestyle=alphabetic,
]{biblatex}
\usepackage[pdftex,
            pdfauthor={B. Dong, H. Graff, J. Mundinger, S. Rothstein, and L. Vescovo},
            pdftitle={Almost All Wreath Product Character Values are Divisible by Given Primes},
            pdfkeywords={representation theory, wreath product}
			]{hyperref}

\hypersetup{
  colorlinks   = true,
  urlcolor     = blue, 
  linkcolor    = blue,
  citecolor   = red
}
\usepackage{amsthm}

\newtheorem{theorem}{Theorem}[section]
\newtheorem{lemma}[theorem]{Lemma}
\newtheorem{claim}[theorem]{Claim}

\newtheorem{proposition}[theorem]{Proposition}
\newtheorem{corollary}[theorem]{Corollary}

\newtheorem*{maintheorem}{Theorem}

\theoremstyle{definition}
\newtheorem{definition}[theorem]{Definition}

\newtheorem*{acknowledgements}{Acknowledgements}

\usepackage{ytableau}

\usepackage{xcolor}

\definecolor{color1}{RGB}{255, 210, 210}
\definecolor{color2}{RGB}{255, 60, 60}
\definecolor{color3}{RGB}{180, 200, 255}
\definecolor{color4}{RGB}{110, 150, 255}

\definecolor{skyeblue}{RGB}{145, 216, 245} 
\definecolor{darkpurp}{RGB}{149, 48, 190}

\newcommand{\tmu}{{\tilde{\mu}}}
\DeclareMathOperator{\Res}{Res}
\title{Almost All Wreath Product Character Values are Divisible by Given Primes}
\author{Brandon Dong}
\address[Brandon Dong]{Carnegie Mellon University, Pittsburgh, PA}
\author{Hannah Graff}
\address[Hannah Graff]{Creighton University, Omaha, NE}
\author{Joshua Mundinger}
\address[Joshua Mundinger]{University of Chicago, Chicago, IL}
\email[Joshua Mundinger]{mundinger@uchicago.edu}
\author{Skye Rothstein}
\address[Skye Rothstein]{Bard College, Annandale-on-Hudson, NY}
\author{Lola Vescovo}
\address[Lola Vescovo]{Macalester College, Saint Paul, MN}

\date{October 27, 2022}


\begin{document}
\begin{abstract}
    For a finite group $G$ with integer-valued character table and a prime $p$, we show that almost every entry in the character table of $G \wr S_N$ is divisible by $p$ as $N \to \infty$. This result generalizes the work of Peluse and Soundararajan on the character table of $S_N$.
\end{abstract}
\maketitle
\section{Introduction}
Let $S_N$ be the symmetric group on $N$ letters. 
The complex irreducible characters of $S_N$ were calculated by Frobenius in 1900; in particular, Frobenius showed that the characters are integer-valued \cite{frobenius}.
In 2019, Alex Miller investigated the distribution of the parity of entries of the character table of $S_N$.
He made the remarkable conjecture that for any prime $p$ and exponent $\ell \geq 1$, the proportion of entries of the character table of $S_N$ divisible by $p$ (and later $p^\ell$ for $\ell \geq 1$) tends to 1 as $N \to \infty$ \cite{miller19, miller19b}. This conjecture was recently proved by Peluse and Soundararajan in the case $\ell=1$ in \cite{peluse2022almost}.

This leaves the question of investigating the distribution of residues modulo $p$ for more general finite groups with integer-valued character tables. A natural infinite family of such is the wreath product $G \wr S_N$ as $N \to \infty$. When $G$ is a fixed group with integer-valued character table, it is known that the characters of $G \wr S_N$ are also integer-valued \cite[Corollary 4.4.11]{james2006representation}.
These families include the Weyl group of type $B_N$, when $G = \mathbb Z/2\mathbb Z$, and wreath products $S_M \wr S_N$ of symmetric groups. 

Our main result is a generalization of Peluse and Soundararajan's theorem:

\begin{maintheorem}[see Theorem \ref{theorem: main theorem} below]
Let $G$ be a group with integer-valued character table and let $G \wr S_N$ be the wreath product of $G$ with $S_N$. For all primes $p$, the proportion of entries in the character table of $G \wr S_N$ which are divisible by $p$ tends to 1 as $N \to \infty$.
\end{maintheorem}

The proof relies on the combinatorics of the representations of $G \wr S_N$. If $G$ has $k$ conjugacy classes, then conjugacy classes and representations of $G \wr S_N$ are both naturally labelled by $k$-multipartitions of $N$. One of the key inputs is characterizing when two elements of $G \wr S_N$ have columns in the character table congruent modulo $p$. In Lemma \ref{lemma: permutation module congruence}, we give a combinatorial characterization directly generalizing the corresponding criterion for $S_N$.

It is known that the character tables of all Weyl groups are integer-valued. The Weyl groups of type $A$ are the symmetric groups, where our question was answered by Peluse and Soundararajan. The Weyl groups of type $B_N$ and $C_N$ are both equal to $\mathbb Z/2\mathbb Z \wr S_N$, handled by our main theorem.  
The only remaining infinite family of Weyl groups is that of type $D$. In Section \ref{section: extensions}, we also show that the proportion of character values of the Weyl group of type $D_N$ divisible by a prime $p$ tends to 1 as $N \to \infty$.

\begin{acknowledgements}
This work was supported by NSF Grant DMS-2149647, and conducted at the MathILy-EST 2022 REU under the direction of Nathan Harman. We thank Sarah Peluse for comments on an earlier version of this paper.
\end{acknowledgements}

\section{Preliminaries}

\subsection{Representation Theory of the Wreath Product}

Let $G$ be a finite group and let $S_N$ be the symmetric group on $N$ letters.

\begin{definition}
The \emph{wreath product of $G$ with $S_N$}, denoted $G \wr S_N$, is the group of $N \times N $ permutation matrices with nonzero entries in $G$.
\end{definition}

We begin by recalling the representation theory of $G \wr S_N$. The representation theory of wreath products was first studied in Specht's dissertation \cite{specht32}, anticipated by Young's work on the case $G = \mathbb Z/2\mathbb Z$ \cite{youngfifth}; see also \cite{zelevinsky,james2006representation} for more modern treatments. 
If we take the representation theory of $G$ as input data and let $N$ vary, the representation theory has structural similarities to the representation theory of $S_N$, the case when $G = 1$. While representations of the symmetric group are labelled by partitions of $N$, representations of the wreath product are labelled by multipartitions:

\begin{definition}
 A \emph{$k$-multipartition} of an integer $N$ is $\lambda = (\lambda_1, \dots, \lambda_k)$ where $\lambda_i$ is a partition for all $i$ such that $\sum_{i=1}^k \vert \lambda_i \vert = N$.
\end{definition}

Suppose that $G$ has $k$ conjugacy classes.
Then $k$-multipartitions of $N$ label the conjugacy classes of $G \wr S_N$. We will not need to use the specific form of this bijection in this paper; it is used in the proofs of character formulas in Propositions \ref{proposition: wreath-irreducibles} and \ref{proposition: perm_modules}, which we omit.

\begin{proposition}[\cite{james2006representation}, Theorem 4.2.8]
\label{proposition: cycle products index conj classes}
If $G$ has $k$ conjugacy classes, then the conjugacy classes of $G \wr S_N$ are indexed by $k$-multipartitions of $N$.
Given $x \in G \wr S_N$, the multipartition $\lambda$ corresponding to $x$ is formed as follows: for each cycle in $x$ of length $\ell$, if the product of the nonzero entries in that cycle is in the $i$th conjugacy class of $G$, then add $\ell$ to $\lambda_i$.
\end{proposition}

One can check the assignment of a conjugacy class to a multipartition is well-defined by checking under conjugation by $S_N$ and by diagonal matrices $G^N \subseteq G \wr S_N$.
Conjugating an element of $G \wr S_N$ by $S_N$ does not change the set of cycle products at all. 
If $(g_1,\ldots,g_N) \in G^N$ and $(12\cdots N)$ is an $N$-cycle,
the conjugate of $(12\cdots N)(g_1,g_2,\ldots,g_N)$ by $(g,1,\ldots,1)$ is $(12\cdots N) (g_1g^{-1},g_2,\ldots,gg_N)$; these two elements have conjugate cycle products $g_N\cdots g_2g_1$ and $g(g_N\cdots g_2g_1)g^{-1}$. The general case of conjugation by $G^N$ reduces to the above case.


To find the complex irreducible representations of $G \wr S_N$, we need the complex irreducible representations of $G$ as input; call the irreducible $G$-representations $V_1, \dots, V_k$.

\begin{proposition}[\cite{james2006representation}, Theorem 4.4.3]
\label{proposition: wreath-irreducibles}
If $G$ has $k$ conjugacy classes, then the irreducible representations of $G \wr S_N$ are in bijection with $k$-multipartitions of $N$. 
For $\lambda=(\lambda_1,\ldots,\lambda_k)$ a $k$-multipartition of $N$, let $a_i = |\lambda_i| $ and $G_a = G \wr S_a$. Then the irreducible representation of $G \wr S_N$ corresponding to $\lambda$ is
\begin{equation*}
    V^{\lambda} = \mathrm{Ind}_{G_{a_1}\times \dots \times G_{a_k}}^{G_N}\left(
        \boxtimes_{i=1}^k \left(
        S^{\lambda_i} \otimes  V_i^{\otimes a_i}\right)
    \right)
\end{equation*}
where $S^{\lambda_i}$ is the Specht module for $S_N$ corresponding to $\lambda_i$.
\end{proposition}

Character values of wreath products can be calculated using a modified version of the Murnaghan-Nakayama rule for the symmetric group. 
Let $\chi^\lambda$ be the character of $V^\lambda$ and $\chi^\lambda_\mu$ be the value of $\chi^\lambda$ on the conjugacy class corresponding to $\mu$.
Then $\chi^\lambda_\mu$ is calculated by decomposing the of Young diagrams of $\lambda_i$ for all $i$ using rimhooks:

\begin{definition}
A \emph{rimhook} of a $k$-multipartition $\lambda=(\lambda_1,\ldots,\lambda_k)$ is $k$ adjacent boxes in the Young diagram of some $\lambda_i$ such that no other boxes are remaining south or east after the rimhook has been removed and no box in the rimhook has a southeast neighbor.
\end{definition}

\begin{figure}[h]
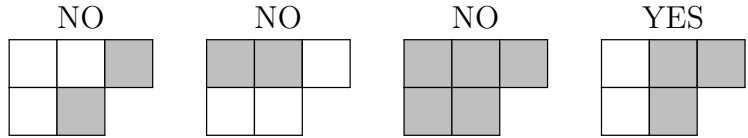

\begin{center}
    
\begin{tabular}{ccccccccc}
    
    NO & & NO & & NO & & YES \\
    \begin{ytableau}
*(white) & *(white) & *(lightgray) \\
	        *(white) & *(lightgray) \\
\end{ytableau} & & \begin{ytableau}
*(lightgray) & *(lightgray) & *(white) \\
	        *(white) & *(white) \\
\end{ytableau} & & \begin{ytableau}
*(lightgray) & *(lightgray) & *(lightgray) \\
	        *(lightgray) & *(lightgray) \\
\end{ytableau} & & \begin{ytableau}
*(white) & *(lightgray) & *(lightgray) \\
	        *(white) & *(lightgray) \\
\end{ytableau}
\end{tabular}

\end{center}

    \caption{
        Examples of three invalid and one valid rimhooks in $\lambda = ((3^12^1))$.
    }
    \label{fig:rimhook example}
\end{figure}

\begin{definition}
For $k$-multipartitions $\lambda$ and $\mu$, a \emph{rimhook decomposition} of $\lambda$ by $\mu$ is obtained by repeatedly removing rimhooks in $\lambda$ with parts of $\mu$ in a fixed ordering such that after all rimhooks have been taken, there are no boxes of $\lambda$ left. All the possible ways to take rimhooks of $\lambda$ with parts of $\mu$ is the set $RHD(\lambda,\mu)$.
\end{definition}

The Murnaghan-Nakayama rule can be modified for wreath products as follows:

\begin{proposition}[\cite{james2006representation}, Theorem 4.4.10]
\label{proposition: MN rule}
Let $\lambda$ and $\mu$ be $k$-multipartitions of $N$. Let $\chi^{1}, \chi^{2}, \ldots,\chi^k$ be the irreducible characters of $G$.
For $\rho \in RHD(\lambda,\mu)$, let $\psi(\rho)$ be defined by
\begin{equation*}
    \psi(\rho) = \prod_{i = 1}^{k}\left({\prod_{\text{rimhooks $h$ in } \lambda_i} \chi^{i} (c_h)}\right)
\end{equation*}
where $c_h$ is the conjugacy class of $G$ associated to $h$. Then
\begin{equation*}
    \chi^\lambda_\mu = \sum_{\rho \in \text{RHD}(\lambda, \mu)} (-1)^{\text{ht}(\rho)} \psi(\rho),
\end{equation*}
where $ht(\rho)$ is the height of the rimhook decomposition.
\end{proposition}

The permutation module characters of wreath products form another basis for the space of class functions of $G \wr S_N$ that is easier to work with.

\begin{definition}\label{definition: permutation modules}
Let $\lambda = (\lambda_1,\ldots,\lambda_k)$ be a $k$-multipartition of $N$ and let $a_i = |\lambda_i|$. 
For each $\lambda_i$, let $S_{\lambda_i}$ be the Young subgroup of $S_{a_i}$ corresponding to $\lambda_i$ and let $G_{\lambda_i} = G \wr S_{\lambda_i}$. 
Then the \emph{permutation module} $M^{\lambda}$ for $G \wr S_N$ is defined by
\begin{equation*}
    M^{\lambda} = \mathrm{Ind}_{G_{\lambda_1} \times \cdots \times G_{\lambda_k}}^{G \wr S_N} 
    \left( \boxtimes_{i=1}^k V_i^{\otimes a_i} \right).
\end{equation*}
\end{definition}

There is a character formula for $M^\lambda$ using row decompositions instead of rimhook decompositions. It is as follows:

\begin{definition}
Let $\lambda$ and $\mu$ be $k$-multipartitions of $N$. A \emph{row decomposition} of $\lambda$ by $\mu$ is 
a function $\rho: \{\text{rows of }\mu\} \to \{\text{rows of }\lambda\}$ such that if $r$ is a row of $\lambda$, then the rows in $\rho^{-1}(r)$ have the same total length as $r$.
The set of all row decompositions of $\lambda$ by $\mu$ is denoted $RD(\lambda, \mu)$. 
\end{definition}

We will think of row decompositions of $\lambda$ by $\mu$ as a tiling of the Young diagrams of $\lambda$ by rows, where rows of $\mu$ are placed in a fixed ordering.

\begin{figure}[h]
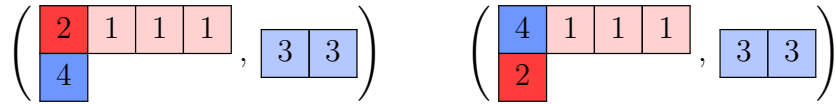

    \centering
    
    \ytableausetup{centertableaux, boxsize = 1.5em}
    \begin{tabular}{ccc}
       $\left( \hspace{1mm} \begin{ytableau}
    *(color2)2 & *(color1)1 & *(color1)1 & *(color1)1 \\
    *(color4)4
    \end{ytableau} \hspace{1mm}, \hspace{1mm} \begin{ytableau}
    *(color3)3 & *(color3)3
    \end{ytableau}
    \right)$ & \hspace{4mm} & $\left( \hspace{1mm} \begin{ytableau}
    *(color4)4 & *(color1)1 & *(color1)1 & *(color1)1 \\
    *(color2)2
    \end{ytableau} \hspace{1mm}, \hspace{1mm} \begin{ytableau}
    *(color3)3 & *(color3)3
    \end{ytableau}
    \right)$
    \end{tabular}
    
    \ytableausetup{boxsize=0.4em}
    \caption{All valid row decompositions of 
    $\left(\ydiagram{4,1},\ydiagram{2}\right)$ by  $({\color{red}{31}}, {\color{blue}{21}})$. 
    The numbers in the boxes indicate the order in which parts of $\mu$ are placed into rows of $\lambda$, with fixed right-to-left placement.
    }
\end{figure}

\begin{proposition}\label{proposition: perm_modules}
Let $\lambda$ and $\mu$ be $k$-multipartitions of $N$.
Let $\chi^1,\chi^2,\ldots,\chi_k$ be the irreducible characters of $G$.
For $\rho \in RD(\lambda,\mu)$, let $\alpha(\rho)$ be defined by
\begin{equation*}
    \alpha(\rho) = \prod_{q = 1}^{k}\left({\prod_{\text{cycles }r \text{ placed into }\lambda_q} \chi^{q} (c_r)}\right),
\end{equation*}
where $c_r$ is the conjugacy class of $G$ associated to $r$.
Then the character for permutation module $M^\lambda$ at $\mu$ is
\begin{equation*}
    M^\lambda_\mu = \sum_{\rho\in \text{RD}(\lambda,\mu)} \alpha(\rho).
\end{equation*}
\end{proposition}
The proof follows from the character formula for induced representations.

We now describe the change-of-basis between irreducible and permutation characters.

\begin{definition}
    The \emph{dominance order} on $k$-multipartitions is defined by $\lambda \succcurlyeq \eta$ if and only if $\lambda_i$ dominates $\eta_i$ for all $i$.
\end{definition}

\begin{lemma} \label{lemma: unitriangularity}
The matrix of multiplicities $[M^\lambda: V^\eta]$ of the irreducible representations of $G \wr S_N$ in permutation modules is unimodular and upper-triangular with respect to dominance order.
\end{lemma}
\begin{proof}

Recall the Kostka numbers $K^{\beta,\gamma}$ for $\beta,\gamma$ partitions of $N$ are defined by
\[ M^\beta = \mathrm{Ind}_{S_\beta}^{S_N} 1 = \bigoplus_\gamma \left(V^{\gamma}\right)^{\oplus K^{\beta,\gamma}},\]
where $S_\beta$ is the Young subgroup corresponding to $\beta$ and $V^\gamma$ is the Specht module corresponding to $\gamma$. Note that our notation for $M^\beta$ and $V^\gamma$ agrees with that of wreath products $G \wr S_N$ when $G=1$.
The Kostka numbers satisfy $K^{\beta,\beta} = 1$ and $K^{\beta,\gamma} > 0$ if and only if $\beta \succcurlyeq \gamma$ in dominance order \cite[I.6]{macdonald98}.

We claim that
\begin{equation} \label{equation: permutation-irrep-decomposition}
    M^\lambda = \bigoplus_{\eta} \left(V^\eta\right)^{\oplus c(\lambda,\eta)}, \qquad c(\lambda,\eta) = \left(\prod_{i = 1}^k K^{\lambda_i, \eta_i} \right).
\end{equation}
By Definition \ref{definition: permutation modules},
if $a_i = |\lambda_i|$ for all $i$ and $H = G_{a_1} \times G_{a_2} \times \cdots \times G_{a_k}$, then 
\[ M^\lambda = Ind_{G_\lambda}^{G_N} \left(\boxtimes_{i=1}^k V_i^{\otimes a_i} \right) = Ind_{H}^{G_N} \left( \boxtimes_{i=1}^k M^{\lambda_i} \otimes V_i^{\otimes a_i}\right),\]
where we make $M^{\lambda_i} \otimes V_i^{\otimes a_i}$ a representation of $G_{a_i}$ by having $S_{a_i}$ act diagonally and $G^{a_i}$ naturally on $V_i^{\otimes a_i}$.
Then \eqref{equation: permutation-irrep-decomposition} follows from multilinearity of the tensor product and linearity of induction. 

Now since the matrix of Kostka numbers is unimodular and upper-triangular with respect to dominance order, the same is true of the matrix $\{c(\lambda,\mu)\}_{\lambda,\mu}$.
\end{proof}

\subsection{Asymptotics of Partitions}

We recall a form of the Hardy-Ramanujan asymptotic for the number of partitions of $N$, denoted $p(N)$.

\begin{proposition}[\cite{HR}, (1.36)] \label{proposition: HR asymptotic}
    If $\delta > 0$, then  
    \begin{equation*}
        \left(\frac{2 \pi}{\sqrt{6}} - \delta\right)\sqrt{N} \leq \log p(N) \leq \left(\frac{2 \pi}{\sqrt{6}} + \delta\right)\sqrt{N}
    \end{equation*}
    for sufficiently large $N$.
\end{proposition}

Let $p_k(N)$ denote the number of $k$-multipartitions of $N$.

\begin{claim}\label{claim: asymptotics of multipartitions}
If $\delta > 0$, then 
\begin{equation*}
    \left(\frac{2 \pi}{\sqrt{6}} - \delta\right) \sqrt{kN} \leq \log{p_k(N)} \leq \left(\frac{2 \pi}{\sqrt{6}} + \delta\right) \sqrt{kN}
\end{equation*}
for sufficiently large $N$.
\end{claim}

This formula also appears in \cite{murty2013partition}. We provide an elementary inductive proof.
\begin{proof}
We proceed by induction on $k$. The base case $k$ = 1 is Proposition \ref{proposition: HR asymptotic}. 

For $\delta > 0$, let $\delta' = \frac{4}{5}\delta$. 
By inductive hypothesis, there exists a constant $B$ such that if $C \geq B$, then 
\begin{equation*}
    \exp\left( \left(\frac{2 \pi}{\sqrt{6}} - \delta'\right) \left( \sqrt{(k-1)C}\right) \right) \ll p_{k-1}(C) \\ \ll \exp\left( \left(\frac{2 \pi}{\sqrt{6}} + \delta'\right) \left(\sqrt{(k-1)C}\right) \right)
\end{equation*}
and 
\begin{equation*}
    \exp\left( \left(\frac{2 \pi}{\sqrt{6}} - \delta'\right) \left( \sqrt{C}\right) \right) \ll p(C) \\ \ll \exp\left( \left(\frac{2 \pi}{\sqrt{6}} + \delta'\right) \left(\sqrt{C}\right) \right).
\end{equation*}
By considering the size of the first partition in a $k$-multipartition, it follows that
\begin{equation*}
    p_k(N) = \sum_{a = 0}^N p(a)p_{k-1}(N-a).
\end{equation*}

We break up the sum for $p_k(N)$ into distinct parts: let
\begin{align*}
    D_1 &= \sum_{a=0}^{B-1} p(a)p_{k - 1}(N - a), \\
    D_2 &= \sum_{a=B}^{N-B} p(a)p_{k - 1}(N - a), \\
    D_3 &= \sum_{a=N-B+1}^N p(a)p_{k - 1}(N - a).
\end{align*} 

In $D_2$, for $B \leq a \leq N - B$, we have
\begin{multline*}
    \exp\left( \left(\frac{2 \pi}{\sqrt{6}} - \delta'\right) \left(\sqrt{a} + \sqrt{(k-1)(N-a)}\right) \right)
    \ll p(a)p_{k - 1}(N - a) \\ 
    \ll \exp\left( \left(\frac{2 \pi}{\sqrt{6}} + \delta'\right) \left(\sqrt{a} + \sqrt{(k-1)(N-a)}\right) \right).
\end{multline*}
Note that $\sqrt{a} + \sqrt{(k - 1)(N - a)} \leq \sqrt{kN}$,
with equality achieved at $a= \frac{N}{k}$. Summing over $a \in [B,N-B]$, we get
\begin{equation}\label{D_2ref}
    \exp\left(\left(\frac{2 \pi}{\sqrt{6}} - \delta\right)\sqrt{kN}\right) \ll D_2 \\ \ll (N - 2B) \exp\left( \left(\frac{2 \pi}{\sqrt{6}} + \delta'\right) \sqrt{kN} \right). 
\end{equation}

We now consider $D_1$ and $D_3$. Note that for $a \in [0, B)$, we have
$p(a)p_{k - 1}(N - a) \leq p(B)p_{k - 1}(N),$ and for $a \in (N - B, N]$, we have $p(a)p_{k - 1}(N - a) \leq p(N)p_{k - 1}(B)$.
Hence 
\begin{align}
     0 \leq D_1 
     \leq B p(B)p_{k - 1}(N) 
     &\ll B
     \exp\left( \left(\frac{2 \pi}{\sqrt{6}} + \delta'\right) \sqrt{kN} \right) \label{D_1ref}
\end{align}
for sufficiently large $N$. Likewise,
\begin{align}
     0 \leq D_3 
     \leq B p(N)p_{k - 1}(B) 
     &\ll B \exp\left( \left(\frac{2 \pi}{\sqrt{6}} + \delta'\right) \sqrt{kN} \right) \label{D_3ref}.
\end{align}
Combining \eqref{D_2ref}, \eqref{D_1ref}, and \eqref{D_3ref}, we have that  
\begin{align}\label{P_k(N)}
    \exp \left(
    {\left(\frac{2 \pi}{\sqrt{6}} - \delta\right)} \sqrt{kN}\right) \ll
    p_k(N) 
    &\ll N\exp \left( {\left(\frac{2 \pi}{\sqrt{6}} + \delta'\right)} \sqrt{kN}\right) \nonumber \\ 
    &\ll \exp \left({\left(\frac{2 \pi}{\sqrt{6}} + \delta\right)} \sqrt{kN}\right) \nonumber
\end{align}
for sufficiently large $N$.
\end{proof}

Claim \ref{claim: asymptotics of multipartitions} implies $k$-multipartitions concentrate around having close to equal-size parts:
\begin{corollary}\label{corollary: concentration of multipartitions}
For all $\delta>0$, the proportion of $k$-multipartitions $\lambda = (\lambda_1, \dots, \lambda_k) \vdash N$ such that 
\[
    \frac{N}{k}(1-\delta) < \vert \lambda_i \vert < \frac{N}{k}(1 + \delta)
\] 
for all $1 \leq i \leq k$ goes to 1 as $N \to \infty$. 

\end{corollary}

\begin{proof}
Pick $0 < \varepsilon < \delta$ and $1\le i \le k$. The number of $k$-multipartitions of $N$ where $|\lambda_i| \notin \left(\frac{N}{k}(1-\varepsilon),\frac{N}{k}(1+\varepsilon)\right)$ is
\begin{equation}\label{equation: badones}
    \sum_{\lambda \text{ s.t. } |\lambda_i| \notin( \frac{N}{k}(1-\varepsilon),\frac{N}{k}(1 + \varepsilon))} p(|\lambda_i|) p_{k-1}(|\lambda_1|, \dots \hat{|\lambda_i|},\dots, |\lambda_k|).
\end{equation}
By Claim \ref{claim: asymptotics of multipartitions}, the rate at which \eqref{equation: badones} approaches infinity is significantly slower than the rate at which $p_k(N)$ approaches infinity. Since $\delta > \varepsilon$, we can conclude that the number of $k$-multipartitions $\lambda$ such that $\vert \lambda_i \vert \in \left(\frac{N}{k}(1 - \delta), \frac{N}{k}(1 + \delta)\right) $ for all $i$ tends to $1$ as $N \to \infty$.
\end{proof}

\section{Main Results}

\subsection{Character Table Column Congruences}

Corollary \ref{corollary: character congruence} below, which we call ``the mashing rule,'' gives a criterion for mod $p$ congruence of two columns of the character table of $G \wr S_N$ in terms of $k$-multipartitions. 

In this section, we must assume that $G$ has integer-valued character table. By \cite[§13.1]{serrelinear}, the group $G$ has integer-valued character table if and only if $\sigma\in G$ is conjugate to $\sigma^j$ whenever $j$ is prime to the order of $\sigma$.

\begin{definition}
Let $\sim_p$ be the equivalence relation on $k$-multipartitions generated by the following: $\mu \sim_p \nu$ if there is $j$ such that $\mu_i = \nu_i$ for $i \neq j$, and $\nu_j$ is formed by replacing one part of size $mp$ in $\mu_j$ with $p$ parts of size $m$ in $\nu_j$.
\end{definition}

\begin{figure}[ht]
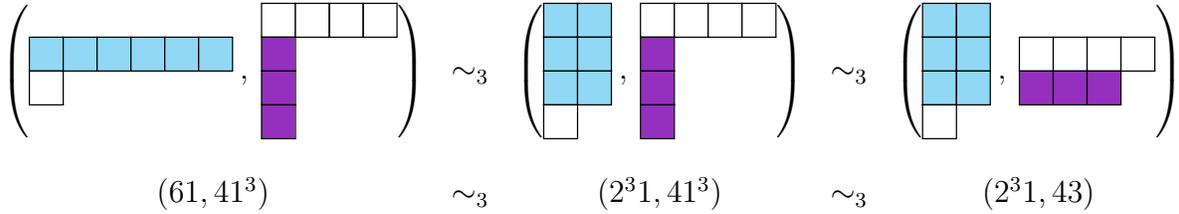

    \centering
    \ytableausetup{boxsize=1.05em} 
    \begin{tabular}{ccccc}
$\left(\begin{ytableau}
	        *(skyeblue) & *(skyeblue) & *(skyeblue) & *(skyeblue) & *(skyeblue) & *(skyeblue)\\
	        *(white)
	   \end{ytableau} \> , \> \begin{ytableau}
	        *(white) & *(white) & *(white) & *(white) \\
	        *(darkpurp) \\
	        *(darkpurp) \\
	        *(darkpurp) \\
	        \end{ytableau} \right)$
	        & $\sim_3$
	        & $\left(\begin{ytableau}
	        *(skyeblue) & *(skyeblue) \\
	        *(skyeblue) & *(skyeblue) \\
	        *(skyeblue) & *(skyeblue) \\
	        *(white)
	   \end{ytableau} \> , \> \begin{ytableau}
	        *(white) & *(white) & *(white) & *(white) \\
	        *(darkpurp) \\
	        *(darkpurp) \\
	        *(darkpurp) \\
	        \end{ytableau} \right)$
	        & $\sim_3$
	        & $\left(\begin{ytableau}
	        *(skyeblue) & *(skyeblue) \\
	        *(skyeblue) & *(skyeblue) \\
	        *(skyeblue) & *(skyeblue) \\
	        *(white)
	   \end{ytableau} \> , \> \begin{ytableau}
	        *(white) & *(white) & *(white) & *(white) \\
	        *(darkpurp) & *(darkpurp) & *(darkpurp)
	        \end{ytableau} \right)$ \\
	         &  &  &  &  \\
	        $(61, 4 1^3)$ & $\sim_{3}$ & $(2^3 1, 4 1^3)$ & $\sim_{3}$ & $(2^3 1, 4 3)$ \\
\end{tabular}
    \caption{Example of three conjugacy classes which are congruent mod 3 in $\mathbb{Z} / 2\mathbb{Z} \wr S_N$ (note that $m=2$ in the first cycle type and $m=1$ in the second).
    }
    \label{figure: mashing example}
\end{figure}

\begin{lemma}\label{lemma: permutation module congruence}
Let $p$ be a prime and $G$ be a group with integer-valued character table. Let $\mu = (\mu_1, ..., \mu_k)$ and $\nu = (\nu_1, ..., \nu_k)$ be $k$-multipartitions of $N$, indexing conjugacy classes of $G \wr S_N$. If $\mu \sim_p \nu$, then $M^\lambda_{ \mu} \equiv M^\lambda_{ \nu} \pmod{p}$ for all $k$-multipartitions $\lambda$ of $N$.
\end{lemma}

\begin{proof}

It suffices to show $M^\lambda_\mu \equiv M^\lambda_\nu \pmod{p}$ if there exists $j$ such that $\mu_i = \nu_i$ for all $i \neq j$, $\mu_j = (\xi, mp)$ for some $\xi$, and $\nu_j = (\xi, m^p)$. 
We break $RD(\lambda,\nu)$ into two cases. In case one, we consider the row decompositions of $\nu$ where $m^p$ is tiled in the same row of $\lambda$. In case two we consider the row decompositions when $m^p$ is not tiled in the same row. Recalling our formula for characters of permutation modules in Proposition \ref{proposition: perm_modules}, let 
\begin{equation}
    \beta = \sum_{\substack{\rho \in RD(\lambda, \nu) \text{ s.t. $m^p$ is tiled} \\ \text{in the same row}}} \alpha(\rho)
\end{equation}
and
\begin{equation}\label{gamma}
    \gamma = \sum_{\substack{\rho \in RD(\lambda, \nu) \text{ s.t. $m^p$ is not tiled} \\ \text{in the same row}}} \alpha(\rho),
\end{equation}
so that $M^\lambda_\nu = \beta + \gamma$. Case one will show $\beta \equiv M^\lambda_\mu \pmod{p}$. Case two shows $\gamma \equiv 0 \pmod{p}$. 
Together, these two congruences imply $M^\lambda_{\mu} \equiv M^\lambda_{\nu} \pmod{p}$.

In both cases, we break into subcases based on the ways to tile $\mu_i$ for $i \neq j$ and $\xi$. In case one, we have compatible tilings for $\mu$ and $\nu$, and in case two, we have additional tilings for $\nu$.

In case one, assume we have tiled all rows of $\mu_i$ for all $i \neq j$ and we have tiled $\xi$. We now have one row remaining. There is only one way to tile the last row for both $\mu$ and $\nu$: put the remaining pieces into the remaining row. Let these row decompositions be denoted $\rho_\mu$ and $\rho_\nu$ respectively.

For $\rho_\mu$, say that we place the final row $r$ of size $mp$ in the partition $\lambda_q$. The associated cycle product is $c_j$ because $mp$ comes from $\mu_j$. Then $mp$ contributes $\chi_q(c_j)$ to the product $\alpha(\rho_\mu)$.
Then for $\rho_\nu$, the $p$ rows of size $m$ are placed into $\lambda_q$. The conjugacy class of $G$ associated with the $p$ rows of size $m$ is again $c_j$, so $m^p$ contributes a factor of $\chi_q(c_j)^p$ to $\alpha(\rho_\nu)$.

By assumption, the character values of $G$ are integral, so by Fermat's little theorem, $\chi_q(c_j) \equiv \chi_q(c_j)^p \pmod{p}$. All other factors in $\alpha(\rho_\mu)$ contributed by $\mu_i$ for $i \neq j$ and $\xi$ are identical to the corresponding factors in $\alpha(\rho_\nu)$ 
Hence, $\alpha(\rho_\mu) \equiv \alpha(\rho_\nu) \pmod{p}$.

Summing over all the tilings in case one, we find $M^\lambda_\mu \equiv \beta \pmod{p}$.

In case two, assume we have tiled all rows of $\mu_i$ for $i \neq j$ and $\xi$, after which there are $t>1$ remaining unfilled rows of the Young diagrams of $\lambda$. If $T \subseteq RD(\lambda,\nu)$ is the set of row decompositions extending our given tiling by $\mu_i$ for $i\neq j$ and $\xi$, then we will show
\[
    \sum_{\rho \in T} \alpha(\rho) \equiv 0 \pmod p.
\]
Then $\gamma$ is the sum over all such $T$ of $\sum_{\rho \in T}\alpha(\rho)$, from which it will follow $\gamma \equiv 0 \mod p$.

Call the lengths of the remaining rows $(m\ell_1 , m\ell_2, \dots, m\ell_t)$.
Since the elements of $T$ are in bijection with choices of placements of $p$ cycles of length $m$ into these rows,
\begin{equation}\label{equation: binomial}
    |T| = \binom{p}{\ell_1, \ell_2, \dots, \ell_t}.
\end{equation}

Let $\rho \in T$. Note that all pieces of $m^p$ come from $\mu_j$, and thus have cycle product $c_j$, while all other cycles in $\mu$ are in the same place in $T$. Thus $\alpha(\rho) = \alpha(\rho')$ for all $\rho,\rho' \in T$. Hence $\sum_{\rho \in T}\alpha(\rho)$ is a sum of $|T|$ identical terms. Then $\sum_{\rho \in T}\alpha(\rho) \equiv 0 \pmod{p}$ because $|T|$ is divisible by $p$.

Case one has shown that $M^\lambda_\mu \equiv \beta \pmod{p}$, and case two has shown that $\gamma \equiv 0 \pmod{p}$. Since $M^\lambda_\nu = \beta + \gamma$, we conclude $M^\lambda_{\mu} \equiv M^\lambda_{\nu} \pmod{p}$. 
\end{proof}

\begin{corollary}[The mashing rule]\label{corollary: character congruence}
Let $G$ have integer-valued character table and $k$ conjugacy classes.
Let $\mu$ and $\nu$ be $k$-multipartitions of $N$. 
If $\mu \sim_p \nu$, then $\chi_{ \mu}^{\lambda} \equiv \chi_{ \nu}^{\lambda} \pmod{p}$ for all irreducible characters $\chi^\lambda$ of $G \wr S_N$.
\end{corollary}

\begin{proof}

The set of irreducible characters and the set of characters of
permutation modules form bases for the space of class functions on $G \wr S_N$. Since the change of basis matrix between these two bases is unimodular and upper-triangular, as stated in Lemma \ref{lemma: unitriangularity}, $\chi^{\lambda}$ can be expressed as an integral linear combination of $M^{\eta}$ for all $k$-multipartitions $\lambda$. It follows from Lemma $\ref{lemma: permutation module congruence}$ that $\mu \sim_p \nu$ implies $\chi_{ \mu}^{\lambda} \equiv \chi_{ \nu}^{\lambda} \pmod{p}$.
\end{proof}

\subsection{Proof of Main Theorem}

Using Corollary \ref{corollary: character congruence}, the existence of one zero in the character table implies many more entries are divisible by $p$. We proceed, following Peluse and Soundararajan in \cite{peluse2022almost}, by using Proposition \ref{proposition: MN rule} to show sufficiently many entries of the character table are zero.

\begin{definition}
A partition is called a \emph{$t$-core} if none of the hook lengths of its Young diagram are divisible by $t$ where $t \in \mathbb{Z}$. For example, from Figure \ref{t-core} one can see that $(4,2,1)$ is a 5-core.
\end{definition}

\begin{figure}[ht]
    \centering
    \ytableausetup{centertableaux}
    \begin{ytableau}
    6 & 4 & 2 & 1 \\
    3 & 1 \\
    1
    \end{ytableau}
    
    \caption{Hook-lengths for $\lambda_i = (4,2,1)$}
    \label{t-core}
\end{figure}

Peluse and Soundararajan proved the following estimate of the number of $t$-cores when $t$ is slightly larger than the typical longest cycle in a random conjugacy class:

\begin{proposition}[\cite{peluse2022almost}, Proposition 1]\label{prop: PSProp1}
Let $L$ be a positive integer, and let $A$ be a real number with $1 \leq A \leq \log L / \log \log L$. Additionally suppose that $t$ is a positive integer with 
\begin{equation} \label{eq: t-bound}
    t \geq \frac{\sqrt{6}}{2 \pi} \sqrt{L} (\log L) \left(1 + \frac{1}{A}\right).
\end{equation}
Then the number of partitions $\lambda$ of $L$ which are not $t$-cores is at most 
\begin{equation*}
    O\left(p(L)\frac{\log L}{L^\frac{1}{2A}}\right),
\end{equation*}
independent of $t$ satisfying \eqref{eq: t-bound}.
\end{proposition}

Complementing the estimate in Proposition \ref{prop: PSProp1}, Peluse and Soundararajan also estimated how many columns of the character table are congruent to a column corresponding to a partition with a large first part:

\begin{proposition}[\cite{peluse2022almost}, Proposition 2]\label{prop: PSProp2}
Let $p \leq\frac{(\log{L})}{(\log{\log{L}})^2}$ be a prime. Starting with a partition $\mu$ of $L$, we repeatedly replace every occurrence of $p$ parts of the same size $m$ by one part of size $mp$ until we arrive at a partition $\tilde{\mu}$ where no part appears more than $p-1$ times. Then the largest part of $\tilde{\mu}$ exceeds
\begin{equation*}
    \frac{\sqrt{6}}{2 \pi} \sqrt{L} \left(\log{L}\right)\left(1+\frac{1}{5p}\right),
\end{equation*}
except for at most
\begin{equation*}
    O\left(p(L)\exp \left( {-L^{\frac{1}{15p}}} \right) \right)
\end{equation*}
partitions $\mu$.
\end{proposition}

We now extend Peluse and Soundarajan's estimate in Proposition \ref{prop: PSProp2} to $k$-multipartitions.

\begin{proposition}\label{proposition: mash to the max}
Let $p\ll N$ be a prime. Given a $k$-multipartition $\mu = (\mu_1, \dots \mu_k)$ of $N$, for all $\mu_i$ with $1 \leq i \leq k$, we repeatedly replace every occurrence of $p$ parts of the same size $m$ by one part of size $mp$ until we arrive at a $k$-multipartition $\tilde{\mu}$ where no part in any $\tilde{\mu}_i$ appears more than $p - 1$ times. 

Then the largest part of $\tilde{\mu}$ is of size at least 
\begin{equation}\label{equation: t lower bound}
    \frac{\sqrt{6}}{2 \pi} \sqrt{\frac{N}{k}} \left(\log{\frac{N}{k}}\right)\left(1+\frac{1}{5p}\right)
\end{equation}
except for a number of multipartitions $\mu$ which is at most 
\begin{equation*}
    O\left(\exp\left(-\left(\frac{N}{k}\right)^\frac{1}{15p}\right) p_k(N)\right).
\end{equation*}
\end{proposition}
\begin{proof}
For a $k$-multipartition $\mu = (\mu_1, \mu_2, \dots \mu_k)$ of $N$, let $\tilde{\mu}$ be as above. We will bound above the number of $k$-multipartitions $\mu$ such that $\tilde{\mu}$ has largest part less than \eqref{equation: t lower bound}.

For any $\mu$, we know that for some $1 \leq i \leq k$, $\vert \mu_i \vert \geq \frac{N}{k}$.
Fix $i$ such that $\mu_i$ has size $\vert \mu_i \vert = a \geq \frac{N}{k}$. Then Proposition \ref{prop: PSProp2} tells us that the largest part of $\tmu_i$ exceeds 
\begin{equation*}
    \frac{\sqrt{6}}{2 \pi} \sqrt{a} \left(\log{a}\right)\left(1+\frac{1}{5p}\right) \geq \frac{\sqrt{6}}{2 \pi} \sqrt{\frac{N}{k}} \left(\log{\frac{N}{k}}\right)\left(1+\frac{1}{5p}\right)
\end{equation*}
except for at most 
\begin{equation*}
    O\left(p(a) \exp\left(-a^\frac{1}{15p}\right) \right)
\end{equation*}
partitions $\mu_i$ of size $a$ and therefore at most 
\begin{equation*}
    O\left(p(a) \exp\left(-a^\frac{1}{15p}\right) p_{k-1}(N - a)\right)
\end{equation*}
total $k$-multipartitions $\mu$ with $\vert \mu_i \vert = a$. Furthermore, since $a \geq \frac{N}{k}$, 
\begin{equation*}
    \exp\left(-a^\frac{1}{15p}\right) \leq \exp\left(-\left(\frac{N}{k}\right)^\frac{1}{15p}\right),
\end{equation*}
and therefore summing over all $ a\geq \frac{N}{k}$ we have that the number of multipartitions $\mu$ such that $\vert \mu_i \vert \geq \frac{N}{k}$ with no part in $\tmu_i$ exceeding \eqref{equation: t lower bound} is at most an absolute constant times

\begin{align*}
    \sum_{a = \frac{N}{k}}^N \exp\left(-a^\frac{1}{15p}\right) p(a)p_{k - 1}(N - a) &\ll \exp\left(-\left(\frac{N}{k}\right)^\frac{1}{15p}\right) \sum_{a = \frac{N}{k}}^N p(a)p_{k - 1}(N - a) \\ 
    &\ll \exp\left(-\left(\frac{N}{k}\right)^\frac{1}{15p}\right) \sum_{a = 0}^N p(a)p_{k - 1}(N - a)  \\ 
    &= \exp\left(-\left(\frac{N}{k}\right)^\frac{1}{15p}\right) p_k(N).
\end{align*}

Since this bound is identical for each $i$, the number of $k$-multipartitions $\mu$ such that $\tilde{\mu}$ does not have a part of size greater than \eqref{equation: t lower bound} is at most a factor of $k$ greater than the bound above, and therefore also at most 

\begin{equation*}
    O\left(\exp\left(-\left(\frac{N}{k}\right)^\frac{1}{15p}\right) p_k(N)\right).
\end{equation*}

\end{proof}

\begin{theorem}\label{theorem: main theorem}
    Let $G$ be a group with integer-valued character table, and let $G \wr S_N$ be the wreath product of $G$ with the symmetric group $S_N$. For all primes $p$, the proportion of entries in the character table of $G \wr S_N$ divisible by $p$ tends to 1 as $N \to \infty$.
\end{theorem}
\begin{proof}
    Let $k$ be the number of conjugacy classes of $G$.
    Given a $k$-multipartition $\mu$, let $\tmu$ be the multipartition obtained by repeatedly replacing $p$ parts of $\mu_i$ size $m$ with one part of size $mp$ until no $\mu_i$ has a part appearing more than $p-1$ times. 
    For $A=5p$, Proposition \ref{proposition: mash to the max} implies that the largest part of $\tmu$ has size 
    \begin{equation}
    \label{equation: final-proof-t-condition}
        t \geq \frac{\sqrt{6}}{2\pi} \sqrt{\frac{N}{k}}\left(\log\frac{N}{k}\right)\left(1 + \frac{1}{A}\right)
    \end{equation}
    for a proportion of $\mu$ tending to 1 as $N \to \infty$.
    Now pick $A' \geq 1$ and $\delta > 0$ such that 
    \[ \left(\log\frac{N}{k}\right)\left(1 + \frac{1}{A}\right) \geq \sqrt{1+\delta} \left(\log\left(\frac{N}{k}(1+\delta)\right)\right) \left(1 + \frac{1}{A'}\right).\]
    By Corollary \ref{corollary: concentration of multipartitions}, the proportion of $k$-multipartitions $\lambda = (\lambda_1,\ldots,\lambda_k)\vdash N$ such that $|\lambda_i| \in \left(\frac{N}{k}(1-\delta),\frac{N}{k}(1+\delta)\right)$ for all $i$ tends to 1 as $N \to \infty$. 
    Thus, consider only $(\lambda,\mu)$ satisfying the above conditions.
    
    Our choice of $\delta$ and $A'$ imply that 
    \[ \frac{\sqrt{6}}{2\pi} \sqrt{\frac{N}{k}}\left(\log\frac{N}{k}\right)\left(1 + \frac{1}{A}\right)
    \geq 
    \frac{\sqrt{6}}{2\pi}\sqrt{\frac{N(1+\delta)}{k}}\left(\log\left(\frac{N}{k}(1+\delta)\right)\right)\left(1 + \frac{1}{A'}\right)
    .\]
    So if $(N_1,\ldots,N_k)$ is a partition of sufficiently large $N$, then by Proposition \ref{prop: PSProp1}, the proportion of $k$-multipartitions $\lambda$ with $|\lambda_i| = N_i$ such that some $\lambda_i$ is not a $t$-core is 
    \[ \sum_{i=1}^k O\left(\frac{\log N_i}{N_i^{\frac{1}{2A'}}}\right) \]
    for all $t$ satisfying \eqref{equation: final-proof-t-condition}, independent of $t$.
    Hence, over all $k$-multipartitions $\lambda$ satisfying $|\lambda_i| \in \left(\frac{N}{k}(1-\delta),\frac{N}{k}(1+\delta)\right)$, the proportion of $\lambda$ such that some $\lambda_i$ is not a $t$-core is 
    \[ O\left( \frac{\log\left(\frac{N}{k}(1+\delta)\right)}{\left(\frac{N}{k}(1-\delta)\right)^{\frac{1}{2A'}}}\right).\]
    
    It follows that most $(\lambda,\mu)$ satisfy that $\lambda_i$ is a $t$-core for $t$ the largest part of $\tmu$. 
    Thus, for a proportion of  $(\lambda,\mu)$ tending to 1 as $N \to \infty$, we have $\chi^\lambda_{\tmu} = 0$ by Proposition \ref{proposition: MN rule} and therefore $\chi^\lambda_\mu \equiv 0\mod p$ by Corollary \ref{corollary: character congruence}.
\end{proof}

\section{Weyl groups of type D}\label{section: extensions}

\begin{definition}
The \emph{Weyl group of type $D_N$} is the group of $N\times N$ signed permutation matrices with an even number of entries equal to $-1$.
\end{definition}
We will denote this group by $D_N$ also (note that it is distinct from the dihedral group).
$D_N$ is a subgroup of $\mathbb{Z}/2\mathbb{Z} \wr S_N$ of index two. Hence, Clifford theory determines its representations:

\begin{proposition}
    The irreducible representations of the Weyl group of type $D_N$ are as follows:
    \begin{enumerate}
        \item if $(\lambda,\mu)$ is a 2-multipartition of $N$ such that $\lambda \neq \mu$, then 
        \[ \Res^{B_N}_{D_N} V^{\lambda,\mu} = \Res^{B_N}_{D_N} V^{\mu,\lambda}\]
        is an irreducible representation of $D_N$;
        \item if $(\lambda,\lambda)$ is a 2-multipartition of $N$ with equal parts, then \[ Res^{B_N}_{D_N} V^{\lambda,\lambda}\]
        is the sum of two irreducible representations of $D_N$.
        \item Each irreducible representation of $D_N$ appears exactly once in (1) or (2).
    \end{enumerate}
\end{proposition}
\begin{proof}
    Let $\psi: B_N \to \{\pm 1\}$ be the character defined by taking the product of the nonzero entries of $B_N$.
    Then $\psi \otimes V^{\lambda,\mu} = V^{\mu,\lambda}$.
    Now the Proposition follows from Clifford theory (see \cite{curtis-reiner-volume-I}).
\end{proof}

\begin{corollary}\label{corollary: D_n}
For all primes $p$, the proportion of entries in the character table of $D_N$ which are divisible by $p$ tends to 1 as $N \to \infty$.
\end{corollary}

\begin{proof}
    The number of irreducible representations of $D_N$ of the form $\Res_{D_N}^{B_N} V^{\lambda,\mu}$ for $\lambda \neq \mu$ equals
    $\frac{1}{2}\left( p_2(N) - p(N/2)\right)$
    when $N$ is even, and $\frac{1}{2}p_2(N)$ when $N$ is odd.
    The number of irreducible representations appearing as a summand of $\Res^{B_N}_{D_N} V^{\lambda,\lambda}$ is $2p(N/2)$ when $N$ is even and $0$ when $N$ is odd.
    By Claim \ref{claim: asymptotics of multipartitions}, we have $p_2(N) \gg p(N/2)$ for large enough $N$, so the proportion of irreducibles of the form $\Res_{D_N}^{B_N} V^{\lambda,\mu}$ goes to 1 as $N \to \infty$.

    Since $D_N \subseteq B_N$ is of index two, at least half of the conjugacy classes in $B_N$ intersect $D_N$. Since most entries in the character table of $B_N$ are divisible by $p$,
    the same is true when we restrict to the columns which intersect $D_N$, since they are at least half of the columns. Hence, the proportion of entries in the character table of $D_N$ which are divisible by $p$ goes to 1 as $N \to\infty$.
\end{proof}

\printbibliography

\end{document}